\documentclass[11pt]{article}
\usepackage[letterpaper, margin=1in]{geometry}
\usepackage{amsmath, amssymb, amsthm, verbatim,enumerate,bbm}
\usepackage{indentfirst}
\usepackage{bbm}
\usepackage{authblk}

\title{Co-degrees resilience for perfect matchings in random hypergraphs}

\author[1]{Asaf Ferber\thanks{Email: asaff@uci.edu. Research is partially supported by an NSF grant 1954395.}}
\author[2]{Lior Hirschfeld\thanks{Email: liorh@mit.edu.}}
\affil[1]{Department of Mathematics, University of California Irvine, Irvine, CA}
\affil[2]{Department of Mathematics, Massachusetts Institute of Technology, Cambridge, MA}

\date{\today}
\allowdisplaybreaks

\theoremstyle{plain}
\newtheorem{theorem}{Theorem}[section]
\newtheorem{lemma}[theorem]{Lemma}

\newtheorem{remark}[theorem]{Remark}
\newtheorem{definition}[theorem]{Definition}

\begin{document}

\maketitle

\begin{abstract}
  In this paper we prove an optimal co-degrees resilience property for the binomial $k$-uniform hypergraph model $H_{n,p}^k$ with respect to perfect matchings. That is, for a sufficiently large $n$ which is divisible by $k$, and $p\geq C_k\log {n}/n$, we prove that with high probability every subgraph $H\subseteq H^k_{n,p}$ with minimum co-degree (meaning, the number of supersets every set of size $k-1$ is contained in) at least $(1/2+o(1))np$ contains a perfect matching.
\end{abstract}

\section{Introduction}

A perfect matching in a $k$-uniform hypergraph $H$ is a collection of vertex-disjoint edges, covering every vertex of $V(H)$ exactly once. Clearly, a perfect matching in a $k$-uniform hypergraph cannot exist unless $k$ divides $n$. From now on, we will always assume that this condition is met.

As opposed to graphs (that is, $2$-uniform hypergraphs) where the problem of finding a perfect matching (if one exists) is relatively simple, the analogous problem in the hypergraph setting is known to be NP-hard (see \cite{Karp}). Therefore, it is natural to investigate sufficient conditions for the existence of perfect matchings in hypergraphs.

A famous result by Dirac \cite{Dirac} asserts that every graph $G$ on $n$ vertices and with minimum degree $\delta(G)\geq n/2$ contains a Hamiltonian cycle (and therefore, by taking alternating edges along the cycle it also contains a perfect matching whenever $n$ is even). Extending this result to hypergraphs provides us with some interesting cases, as one can study `minimum degree' conditions for subsets of any size $1\leq \ell <k$. That is, given a $k$-uniform hypergraph $H=(V,E)$ and a subset of vertices $X$, we define its \emph{degree}
$$d(X)=|\{e\in E :X\subseteq e\}|.$$ Then, for every $1\leq \ell<k$ we define
$$\delta_{\ell}(H)=\min \{d(X) : X\subseteq V(H), |X|=\ell\},$$
to be the \emph{minimum $\ell$-degree of $H$}. A natural question is: Given $1\leq \ell<k$, what is the minimum $m_{\ell}(n)$ such that every $k$-uniform hypergraph on $n$ vertices with $\delta_\ell(H)\geq m_{\ell}(n)$ contains a perfect matching?

The above question has attracted a lot of attention in the last few decades. For more details about previous work and open problems, we will refer the reader to surveys by R\"{o}dl and Ruci\'{n}ski \cite{RodlRucinski} and Keevash \cite{Keevash}. In this paper we restrict our attention to the case where $\ell=k-1$. Following a long line of work in studying this property, which is expanded upon in the former survey,  K{\"u}hn and Osthus proved in \cite{DK} that every $k$-uniform hypergraph with $\delta_{k-1}\geq n/2+\sqrt[]{2n \log n}$ contains a perfect matching. This bound is optimal with an additive error term of $\sqrt[]{2n \log n}$. Note that one can view this result as follows: Start with a complete $k$-uniform hypergraph on $n$ vertices (this clearly contains a perfect matching). Imagine that an adversary is allowed to delete `many' edges in an arbitrary way, under the restriction that he/she cannot delete more than $r$ edges that intersect on a subset of size at least $(k-1)$. What then, is the largest $r$ for which the resulting hypergraph always contains a perfect matching? We refer to this value as the `$(k-1)$-local-resilience' of the hypergraph. The above mentioned result equivalently shows that such a hypergraph has `$(k-1)$-local-resilience' at least $n/2-\sqrt[]{2n \log n}$.


Here we study a similar problem in the random hypergraph setting. Let $H^k_{n,p}$ be a random variable which outputs a $k$-uniform hypergraph on vertex set $[n]$ by including any $k$-subset $X\in \binom{[n]}{k}$ as an edge with probability $p$, independently. The existence of perfect matchings in a typical $H^k_{n,p}$ is a well studied problem with a very rich history. Unlike for random graphs where finding a `threshold' for the existence of a perfect matching is quite simple, the problem of finding a `threshold' function $p$ for the existence of a perfect matching, with high probability, in the hypergraph setting is notoriously hard. After a few decades of study, in 2008 Johansson, Kahn and Vu \cite{JKV} finally managed to determine the threshold. Among their results, one of particular note is that for $p\geq C\log n/n^{k-1}$, whp $H^k_{n,p}$ contains a perfect matching. On the other hand, it is quite simple to show that if $p\leq c\log n/n^{k-1}$ for some small constant $c$, then a typical $H^k_{n,p}$ contains isolated vertices and thus has no perfect matchings.

In this note we determine the `$(k-1)$-local-resilience' of a typical $H^k_{n,p}$. Note that if $p=o(\log n/n)$ then whp there exists a ($k-1$)-set of vertices which is not contained in any edge and therefore, for the study of $(k-1)$-resilience, it is natural to restrict our attention to $p\geq C\log n/n$ (which is significantly above the threshold for a perfect matching as obtained in \cite{JKV}). The following theorem gives a complete solution to this problem for this range of $p$.

\begin{theorem}
  \label{main:codegree} Let $k\in \mathbb{N}$, let $\varepsilon>0$, and let $C:=C(k,\varepsilon)$ be a sufficiently large constant. Then, for all $p\geq \frac{C\log n}{n}$, whp a hypergraph $H^{k}_{n,p}$ is such that the following holds:
  Every spanning subhypergraph $H\subseteq H^{k}_{n,p}$ with $\delta_{k-1}(H)\geq (1/2+\varepsilon)np$ contains a perfect matching.
\end{theorem}

Next, we show that the above theorem is asymptotically tight.

\begin{theorem}
  Let $k\in \mathbb{N}$, let $\varepsilon>0$, and let $C:=C(k,\varepsilon)$ be a sufficiently large constant. Then, for all $p\geq \frac{C\log n}{n}$, any hypergraph $H^{k}_{n,p}$ is such that the following holds: Whp there exists $H\subseteq H^k_{n,p}$ with $\delta_{k-1}(H)\geq (1/2-\varepsilon)np$ that does not contain a perfect matching.
\end{theorem}

\begin{proof} [Sketch]
This proof is based on an idea of K\"uhn and Osthus outlined in \cite{DK}. Fix a partition of $V(H)=V_1\cup V_2$ into two sets of size roughly $n/2$, where $|V_1|$ is odd. Now, expose all the edges of $H^k_{n,p}$ and let $H$ be the subhypergraph obtained by deleting all the hyperedges that intersect $V_1$ on an odd number of vertices. Clearly, $H$ cannot have a perfect matching, as every edge covers an even number of vertices in $V_1$ and $|V_1|$ is odd. Now, we demonstrate that every $(k-1)$-subset of vertices still has at least $(1/2-\varepsilon)np$ neighbors in $H$. Indeed, given any $(k-1)$ subset $X$, we distinguish between two cases:

1. $|X\cap V_1|$ is even -- as we clearly kept all the edges of the form $X\cup \{v\}$, $v\in V_2$, and since $|V_2|\approx n/2$, by a standard application of Chernoff's bounds, $X$ is contained in at least $(1/2-\varepsilon)np$ many such edges as required.

2. $|X\cap V_1|$ is odd -- as we clearly kept all the edges of the form $X\cup \{v\}$, $v\in V_1$, and since $|V_1|\approx n/2$, a similar reasoning as in 1. gives the desired.

All in all, whp the resulting subhypergraph has $\delta_{k-1}(H)\geq (1/2-\varepsilon)np$ and does not contain a perfect matching.
\end{proof}

\section{Notation}
For the sake of brevity, we present the following, commonly used notation:

Given a graph $G$ and $X\subseteq V(G)$, let $N(X)=\cup_{x\in X}N(x)$. For two subsets $X,Y\subseteq V(G)$ we define $E(X,Y)$ to be the set of all edges $xy\in E(G)$ with $x\in X$ and $y\in Y$, and set $e_G(X,Y):=|E(X,Y)|$.

For a $k$-uniform hypergraph $H$ on vertex set $V(H)$, and for two subsets $X,Y\subseteq V(H)$ we define
$$d(X,Y)=|\{e\in E(H): X\subseteq e \text{ and } e\setminus X\subseteq Y\}|.$$

Given any $k$-partite, $k$-uniform hypergraph with parts $V(H)=V_1\cup \ldots \cup V_{k}$ of the same size $m$ we consider all $V_i$ to be disjoint copies of the integers $1$ to $m$, without loss of generality.

Finally, for every random variable $X$, we let $M(X)$ be its \emph{median}.

\section{Outline}

In this section we give a brief outline of our argument. Consider a typical $H^k_{n,p}$, and let $H\subseteq H^k_{n,p}$ with $\delta_{k-1}(H)\geq (\frac{1}{2}+\varepsilon)np$. In order to show that $H$ contains a perfect matching, we first show that some auxiliary bipartite graph $B$ contains a perfect matching. Then, we show that every perfect matching in $B$ can be translated into a perfect matching in $H$.

To this end, we first find a partition $V(H)=V_1 \cup \cdots \cup V_k$, with all $V_i$'s having the exact same size $m = \frac{n}{k}$, such that the following property holds: For every subset $X\in \binom{[n]}{k-1}$ and for every $1\leq i\leq k$ we have

$$d_H(X,V_i)\in (1\pm \varepsilon)\cdot \frac{d_{H}(X)}{k}.$$

Then, we let $H'$ be the $k$-partite, $k$-uniform subhypergraph induced by this partition of $V(H)$.

Now, given some set of permutations $\pi = \{ \pi _1, \pi_2, \cdots, \pi _{k-1}$ \}, $\pi _i = [m] \rightarrow V_i$, we can construct a bipartite graph $B_\pi(H')$ as follows:

The parts of $B_{\pi}(H')$ are $V_k$ and
$$X_\pi=\{\{\pi_1(i),\pi_2(i),\ldots,\pi_{k-1}(i)\}\mid 1\leq i \leq m\}.$$  

The edges of $B_{\pi}(H')$ consist of all pairs $xv\in X_\pi \times V_k$, for which $x\cup \{v\}\in E(H')$.

A moment's thought now reveals that a perfect matching in any such $B_\pi(H')$ corresponds to a perfect matching in $H'$, which itself corresponds to a perfect matching in $H$. Therefore, the main part of the proof consists of showing that, with high probability, there exists a $\pi$ such that $B_\pi(H')$ contains a perfect matching.

\section{Tools and Preliminary Results}

In this section we present some tools to be used in the proof of our main result.

\subsection{Chernoff's inequalities}

First, we need the following well-known bound on the
  upper and lower tails of the binomial distribution, outlined by Chernoff
  (see Appendix A in \cite{AlonSpencer}).

    \begin{lemma}[Chernoff's inequality]\label{lemma:Chernoff}
    Let $X \sim Bin(n, p)$ and let
    ${\mathbb E}(X) = \mu$. Then
      \begin{itemize}
        \item
        ${\mathbb P}\left ( X < (1 - a)\mu \right) < e^{-a^2\mu /2}$
        for every $a > 0$;
        \item ${\mathbb P} \left( X > (1 + a)\mu \right) <
        e^{-a^2\mu /3}$ for every $0 < a < 3/2$.
      \end{itemize}
    \end{lemma}

\begin{remark}\label{rem:hyper} These bounds also hold when $X$ is
  hypergeometrically distributed with
  mean $\mu $.
\end{remark}

In addition, we will make use of the following simple bound.

\begin{lemma}
  \label{lemma:fancy Chernoff}
Let $X\sim Bin(m,q)$. Then, for all $k$ we have
$$\Pr[X\geq k]\leq \left(\frac{emq}{k}\right)^k.$$
\end{lemma}

\begin{proof}
  Note that
  $$\Pr[X\geq k]\leq \binom{m}{k}q^k\leq \left(\frac{emq}{k}\right)^k$$
  as desired.
\end{proof}

\subsection{Talagrand's type inequality}

Our main concentration tool is the following theorem from McDiarmid \cite{McD}.

\begin{theorem}
  \label{McDiarmid}
Given a set $S$ of size $m$, we let $Sym(S)$ denote the set of all $m!$ permutations of $S$. Let $\{ B_1,\ldots,B_k \}$ be a family of finite non-empty sets, and let $\Omega=\prod_i Sym(B_i)$. Let $\pi= \{ \pi_1,\ldots,\pi_k \}$ be a family of independent permutations, such that for $i$, $\pi_i\in Sym(B_i)$ is chosen uniformly at random.

Let $c$ and $r$ be constants, and suppose that the nonnegative real-valued function $h$ on $\Omega$ satisfies the following conditions for each $\pi\in \Omega$.

\begin{enumerate}
  \item Swapping any two elements in any $\pi_i$ can change the value of $h$ by at most $2c$.
  \item If $h(\pi)=s$, there exists a set $\pi _{proof} \subseteq \pi$ of size at most $rs$, such that $h(\pi ') \geq s$ for any $\pi'\in \Omega$ where $ \pi ' \supseteq \pi _{proof}$.
\end{enumerate}

Then for each $t\geq 0$ we have
  $$\Pr[h\leq M(h(\pi))-t]\leq 2\exp\left(-\frac{t^2}{16rc^2M}\right).$$
\end{theorem}

\subsection{Hall's theorem}

It is convenient for us to work with the following equivalent version of Hall's theorem (the proof is an easy exercise).
\begin{theorem}
  \label{thm:Hall}
Let $G=(A\cup B,E)$ be a bipartite graph with $|A|=|B|=n$. Then, $G$ contains a perfect matching if and only if the following holds:
\begin{enumerate}
  \item For all $X\subseteq A$ of size $|X|\leq n/2$ we have $|N(X)|\geq |X|$, and
  \item For all $Y\subseteq B$ of size $|Y|\leq n/2$ we have $|N(Y)|\geq |Y|$.
\end{enumerate}
\end{theorem}

\subsection{Properties of random hypergraphs}

In this section we collect some properties that a typical $H^k_{n,p}$ satisfies. First, we show that all the $(k-1)$-degrees are `more or less' the same.

\begin{lemma}
  \label{degrees in random}
  Let $\varepsilon>0$ and let $k\geq 2$ be any integer. Then, whp we have
  $$(1-\varepsilon)np\leq \delta_{k-1}(H^k_{n,p})\leq \Delta_{k-1}(H^k_{n,p})\leq (1+\varepsilon)np,$$
  provided that $p=\omega(\log n/n)$.
\end{lemma}

\begin{proof}
  Let us fix some $X\in \binom{[n]}{k-1}$. Observe that $d(X)\sim Bin(n-k+1,p)$, and therefore $$\mu:=\mathbb{E}[d(X)]=(n-k+1)p.$$
  Hence, by Chernoff's inequalities we obtain that
  $$\Pr[d(X)\notin (1\pm \varepsilon)\mu]\leq 2e^{-\frac{\varepsilon^2\mu}{3}}=o(1/n^k).$$
  All in all, by taking a union bound over all sets $\binom{[n]}{k-1}$, we conclude that
  $$\Pr[\exists X\in \binom{[n]}{k-1} \text{ s.t. } d(X)\notin (1\pm\varepsilon)\mu]=o(1).$$
  This completes the proof.
\end{proof}

In the proof of our main result we will convert the problem of finding a perfect matching in $H$ into the problem of finding a perfect matching in some auxiliary bipartite graph. In order to do so, we wish to partition our hypergraph $H\subseteq H^k_{n,p}$ into $k$ equal parts satisfying some `degree assumptions', and then to define our auxiliary bipartite graph based on such a partition. In the following lemma we show that, given a $k$-uniform hypergraph $H$ with `relatively large' $(k-1)$-degree, a random partition of its vertices into equally sized parts satisfies these assumptions.

\begin{lemma}
  \label{random partition}
  For every $\varepsilon>0$ there exists $C:=C(\varepsilon)$ for which the following holds. Let $H$ be a $k$-uniform hypergraph on $n$ vertices, where $n$ is sufficiently large. Suppose that $\delta_{k-1}(H)\geq C\log n$ and that $n$ is divisible by $k$.
  Then, there exists a partition $V(H)=V_1\cup \ldots \cup V_{k}$ into sets of the exact same size satisfying the following property: For every subset $X\in \binom{[n]}{k-1}$ and for every $1\leq i\leq k$ we have $$d_H(X,V_i)\in (1\pm \varepsilon)\cdot \frac{d_{H}(X)}{k}.$$

\end{lemma}

\begin{proof}
Let $H$ be a a $k$-uniform hypergraph on $n$ vertices, where $n$ is sufficiently large. Consider the random partition $V(H)=V_1\cup \ldots \cup V_{k}$ into sets of the exact same size. For some fixed $X$ and $i$, observe that $d_{H}(X,V_i)$ is hypergeometrically distributed with an expected value of $\frac{d_{H}(X)}{k}$. Therefore, we can use Lemma \ref{lemma:Chernoff} to determine that

$$\Pr[d_{H}(X,V_i)>(1+ \varepsilon)\cdot \frac{d_{H}(X)}{k}]\leq e^{-\varepsilon^2 \frac{d_{H}(X)}{k}/3}\leq e^{-k\log n}=n^{-k},$$

where the last inequality holds for a large enough $C$.

By applying a union bound over all possible $X$'s and $i$'s, we obtain that the probability of having such a set and an index $i$ is at most

$$\binom{n}{k-1}kn^{-k}=o(1).$$

Similarly, we obtain that

$$\Pr\left[\exists X \text{ and }i: d_{H}(X,V_i)<(1- \varepsilon)\cdot \frac{d_{H}(X)}{k}\right]=o(1).$$

This completes the proof.
\end{proof}

\begin{definition}
  \label{pseudorandom}
  Let $\varepsilon>0$, $p\in (0,1]$, and $m\in \mathbb{N}$. A bipartite graph $G=(A\cup B,E)$ with $|A|=|B|=m$ is called $(\varepsilon,p)$-\emph{pseudorandom} if it satisfies the following properties:
  \begin{enumerate}
    \item $\delta(G)\geq (1/2+\varepsilon)mp$,
    \item  for every $X\subseteq A$ and $Y\subseteq B$ with $|X|-1=|Y|\leq m/10$ we have $e_G(X,Y)\leq mp|X|/2$,
    \item  for every $X\subseteq A$ and $Y\subseteq B$ with $m/10\leq |X|-1=|Y|\leq m/2$ we have $e_G(X,Y)\leq (1/2+\varepsilon/2)mp|X|$
  \end{enumerate}
\end{definition}

\begin{definition}
  \label{auxiliary-bipartite}
  Let $H'$ be a $k$-partite, $k$-uniform hypergraph with parts $V(H')=V_1\cup \ldots \cup V_{k}$ of the same size $m$. Given a set of permutations $\pi=\{\pi_1,\pi_2,\ldots\pi_{k-1}\}$, $\pi_i:[m]\rightarrow V_i$, we construct an auxiliary bipartite graph, $B_\pi:=B_{\pi}(H')$, as follows:

Let $X_\pi=\{\{\pi_1(i),\pi_2(i),\ldots,\pi_{k-1}(i)\};1\leq i \leq m\}$ and $V_k$ be the parts of $B_\pi$. For every pair $xv$ with $x\in X_\pi$ and $v\in V_k$, we let $xv\in E(B_{\pi})$ iff $x\cup \{v\}\in E(H')$.
\end{definition}

\begin{remark}\label{remark:aux}
  Note that every edge in a given $B_{\pi}(H')$ with parts $x\in X_\pi$ and $v\in V_k$ corresponds to an edge $\pi_1 (i) \cup \pi_2 (i)\ldots \pi_{k-1} (i) \cup \{v\}$ in $H'$ for  some $1 \leq i \leq m$. Therefore, if $B_{\pi}(H')$ contains a perfect matching, clearly $H'$ contains a perfect matching as well. Having established this fact, our main goal is to show that there exists a $\pi$ for which $B_{\pi}$ contains a perfect matching.
\end{remark}

We now wish to demonstrate that given a `proper' $k$-partite, $k$-uniform hypergraph $H'$, a randomly chosen $\pi$ results in a $B_{\pi}(H')$  with a sufficiently large minimum degree. As will be seen soon, the `problematic' random variables that we need to control are $d_{B_{\pi}}(v)$, where $v\in V_k$. In order to prove that these variables concentrate about their expectation, we will use Theorem \ref{McDiarmid}.



For the sake of simplicity in the following lemma, we define this notation: Suppose that $H'$ is a $k$-partite, $k$-uniform hypergraph with parts $V(H')=V_1\cup \ldots \cup V_k$. Let $W_i:=V_1\times \ldots V_{i-1}\times V_{i+1}\times \ldots \times V_k$. For every $X\in W_i$ (note that $|X|=k-1$) define
$$\delta^*_{k-1}(H'):=\min\{d(X,V_i): X\in W_i, \text{ and }1\leq i\leq k \}.$$

\begin{lemma}
	\label{replace M}
    Let $0<\alpha<1/2$ and let $m\in \mathbb{N}$ be sufficiently large. Let $H'$ be a $k$-partite, $k$-uniform hypergraph with parts $V(H')=V_1\cup \ldots \cup V_{k}$ of the same size $m$. Suppose that $\delta^*_{k-1}(H')\geq 200/\alpha^2$. Let $B_\pi$ be the auxiliary-bipartite graph formed from the set of permutations $\pi:=\{\pi _1,id_2,..., id_{k-1}\}$, where $\pi _1$ is a random permutation of $V_1$ and each $id_j$ is the identity permutation of $V_j$. Let $\mu _v = \mathbb{E}[d_{B_{\pi}}(v)]$. Then, for every $v \in V_k$ we have
    $$M _v = M(d_{B_{\pi}}(v)) \in (1\pm \alpha)\mu _v.$$
\end{lemma}
\begin{remark}
  The above lemma enables us to use $\mu _v$ instead of $M _v$ in Theorem \ref{McDiarmid} when it is applied to $d_{B_{\pi}}(v)$.
\end{remark}

\begin{proof}
	Consider the $B_\pi$, formed from the set of permutations $\pi:=\{\pi_1,id_2,..., id_{k-1}\}$, where $\pi _1$ is a random permutation of $V_1$ and each $id_j$ is the identity permutation of $V_j$. Let $v$ be some element in $V_k$. For each $1 \leq i \leq m$, let $A_i := \{{id_2 (i)}, {id_3 (i)}\ldots, {id_{k-1} (i)}\}$, and let $d_i(v)$ be the number of extensions of $\{v\} \cup A_i$ into $V_1$ (that is, the number of edges $e\in E(H')$ for which $\{v\}\cup A_i\subseteq e$). Moreover, let $d_v = \sum_i d_i(v)$, and for each $i$ define a indicator random variable $\mathbbm{1}_i$, where $\mathbbm{1}_i = 1$ if $\{\pi_1(i)\} \cup A_i \cup \{v\} \in E(H')$. Observe that $d_{B_{\pi}}(v)=\sum \mathbbm{1}_i$.

Our plan is to compute $\mu _v:=\mathbb{E}[d_{B_{\pi}}(v)]$ and $\sigma^2=Var(d_{B_{\pi}}(v))$ and to show that $\sigma^2\leq \alpha^2\mu _v^2/100$. The desired result will then be easily obtained as follows: First, note that by Chebyshev's inequality  we have $$\mathbb{P}[|d_{B_\pi} (v) - \mu _v| \geq \alpha \mu _v] \leq \frac{\sigma^2}{\alpha ^2 \mu_v^2}\leq 1/100.$$
Since with probability at least $99/100$ we have that $d_{B_{\pi}}(v)\in (1\pm \alpha)\mu_v$, we conclude that the median also lies in this interval.

It remains to compute $\mu_v$ and $\sigma^2$. Since $\mathbb{P}[\mathbbm{1}_i = 1] = \frac{d_i (v)}{m}$, by linearity of expectation we obtain

$$\mu _v = \sum_{i=1}^m \mathbb{E}[\mathbbm{1}_i] = \sum_{i=1}^m \frac{d_i (v)}{m} = \frac{d_v}{m}.$$

To compute the variance, note that

 \begin{align*}
Var\left(d_{B_\pi}(v)\right) &=Var\left(\sum_{i=1}^{m} \mathbbm{1}_i\right) = \sum_{i=1}^m Var\left(\mathbbm{1}_i\right) + 2 \sum_{i<j} Cov(\mathbbm{1}_i, \mathbbm{1}_j)\\
&\leq \mu _v + 2\sum_{i<j}\left(\mathbb{E}[\mathbbm{1}_i\mathbbm{1}_j]-\mathbb{E}[\mathbbm{1}_i]\mathbb{E}[\mathbbm{1}_j]\right)\\
&\leq \mu _v + 2\sum_{i<j}\left(\frac{d_i(v) d_j(v)}{m(m-1)} - \frac{d_i(v) d_j(v)}{m^2}\right) = \mu _v + 2\sum_{i<j}\left(\frac{d_i(v) d_j(v)}{m^2 (m-1)}\right)\\
&\leq \mu _v + 2\sum_{i=1}^m \sum_{j=1}^m\left(\frac{d_i(v) d_j(v)}{m^2 (m-1)}\right)\leq \mu _v + 2\sum_{i=1}^m\left(\frac{d_i(v) d_v}{m^2 (m-1)}\right)\\
& = \mu _v + \frac{2  d_v^2}{m^2 (m-1)}=\mu_v+\frac{2\mu_v^2}{m-1}.
\end{align*}

To complete the proof let us first observe that since $m$ is sufficiently large we have $\frac{2\mu_v^2}{m-1}\leq \alpha^2\mu_v^2/200$. Second, note that since $\mu_v\geq 200/\alpha^2$ we have that $\mu_v\leq \alpha^2\mu_v^2/200$. Plugging these estimates into the last line of the above equation gives us the desired.
\end{proof}

\begin{lemma}
  \label{Lemma:HighDegreeBipartite}
   For every $\varepsilon>0$ there exists $C:=C(\varepsilon)$ for which the following holds for sufficiently large $m\in \mathbb{N}$ and $p = C\log m / m$. Let $H'$ be a $k$-partite, $k$-uniform hypergraph with parts $V(H')=V_1\cup \ldots \cup V_{k}$ of the same size $m$. Suppose that $\delta^*_{k-1}(H')\geq(\frac{1}{2}+\varepsilon)mp$. Then there exists $\pi:=\{\pi_1,\pi_2,...\pi_{k-1}\}$, $\pi_i:[m]\rightarrow V_i$, s.t. $\delta(B_\pi)\geq (\frac{1}{2} + \frac{\varepsilon}{2})mp$.
\end{lemma}

\begin{proof}

Consider the $B_\pi$, formed from the set of permutations $\pi:=\{\pi_1,id_2,..., id_{k-1}\}$, where $\pi_1$ is random and $id_j$ is the identity permutation for $V_j$. As $\delta^*_{k-1}(H')\geq(\frac{1}{2}+\varepsilon) mp$, it is guaranteed that for all $x\in X_\pi$ we have (deterministically) that $d_{B_\pi}(x)\geq(\frac{1}{2}+\varepsilon) mp$.

Consider some $v\in V_k$ and observe from the proof of Lemma \ref{replace M}, under the same notation, that
$\mathbb{E}[d_{B_{\pi}}(v)]=\frac{d_v}{m}\geq (1/2+\varepsilon)mp.$

In order to complete the proof, we want to show that the $d_{B_{\pi}}(v)$'s are `highly concentrated' using Theorem \ref{McDiarmid}. To this end, let $h(\pi)=d_{B_\pi}(v)$ and note that swapping any two elements of $\pi_1$ can change $h$ by at most $2$. Moreover, note that if $h(\pi)\geq s$, then it is enough to specify only $s$ elements of $V$. Therefore, $h(\pi)$ satisfies the conditions outlined by Talagrand's type inequality with $c=1$ and $r=1$.

Now, let $\alpha=\varepsilon/100$, and observe that by Lemma \ref{replace M} we have that the median $M$ of $d_{B_{\pi}}(v)$ lies in the interval $(1\pm \alpha)\mathbb{E}[d_{B_{\pi}}(v)].$

Therefore, we have
  $$\Pr[h\leq (\frac{1}{2} + \varepsilon/2)mp]\leq \Pr[h\leq (1-\varepsilon/2)\mathbb{E}[d_{B_{\pi}}(v)]]$$ and the latter is at most
  $$\Pr[h\leq (1-\varepsilon/2)(1+\alpha)M]\leq \Pr[h\leq (1-\varepsilon/4)M].$$

  Now, by Theorem \ref{McDiarmid} we obtain that

  $$\Pr[h\leq (1/2 + \varepsilon/2)mp]\leq 2\exp\left(-\frac{(\varepsilon M/4)^2}{16M}\right).$$

Next, using (again) the fact that $M\in (1\pm \alpha)\mathbb{E}[d_{B_{\pi}}(v)]$ and that $\mathbb{E}[d_{B_{\pi}}(v)]=\Theta(mp)\geq C \log m$, we can upper bound the above right hand side by
  $$2\exp\left(-\Theta(mp)\right)\leq n^{-2}.$$

Finally, in order to complete the proof, we take a union bound over all $v\in V_k$ and obtain that whp $\delta(B_\pi)\geq (\frac{1}{2} + \frac{\varepsilon}{2})mp$.
\end{proof}

\begin{lemma}
 \label{Lemma:AlwaysPseudorandom}Let $\varepsilon>0$, $k\in \mathbb{N}$ and $p\geq C\log n/n$, where $C:=C(\varepsilon,k)>0$ is a sufficiently large constant. Then, a random hypergraph $H^k_{n,p}$ with high probability satisfies the following:
 For every $k$-partite, $k$-uniform subhypergraph $H' \subseteq H_{n,p}^k$ with parts $V(H')=V_1\cup \ldots \cup V_{k}$ of the same size $m:=\frac{n}{k}$, if $\delta^*_{k-1}(H')\geq (1/2+\varepsilon)mp$, there exists $\pi:=\{\pi_1,\pi_2,...\pi_{k-1}\}$, $\pi_i:[m]\rightarrow V_i$, s.t. $B_\pi$ is $(\varepsilon / 2,p)$-pseudorandom.
\end{lemma}

\begin{proof}
  Let $H'$ be such a subhypergraph. Our goal is to prove the existence of $\pi$ for which $B_{\pi}$ is $(\varepsilon /2,p)$-pseudorandom. That is, we want to show that $B_{\pi}$ satisfies the following properties:
  \begin{enumerate}
    \item $\delta(B_{\pi})\geq (1/2+\varepsilon / 2)mp$,
    \item  for every $X\subseteq X_\pi$ and $Y\subseteq V_k$ with $|X|-1=|Y|\leq m/10$ we have $e_{B_{\pi}}(X,Y)\leq mp|X|/2$,
    \item  for every $X\subseteq X_\pi$ and $Y\subseteq V_k$ with $m/10\leq |X|-1=|Y|\leq m/2$ we have $e_{B_{\pi}}(X,Y)\leq (1/2+\varepsilon/4)|X|mp$
  \end{enumerate}

  Let $\pi$ be obtained as in Lemma \ref{Lemma:HighDegreeBipartite}, and consider $B_{\pi}=(X_\pi \cup V_k,E)$. Clearly, Property $1$ is satisfied by the conclusion of Lemma \ref{Lemma:HighDegreeBipartite}.

  For Property $2$, let us fix $X\subseteq X_\pi$ and $Y\subseteq V_k$ of sizes $x$ and $y$ respectively where $x-1=y\leq m/10$. We now wish to establish an upper bound for the number of edges between them. Assume towards contradiction that $e_{B_{\pi}}(X,Y)>mpx/2$. Observe that this translates to the following: There exist $x$ disjoint sets $F_1,\ldots,F_x$, each of size exactly $k-1$ and a set $Y$ of size $x-1$, which is disjoint to all the $F_i$s, such that the number of edges in $H^k_{n,p}$, of the form $F_i\cup \{a\}$ where $a\in Y$, is larger than $mpx/2$. Let us show that whp $H^k_{n,p}$ has no such sets, thereby also guaranteeing that whp no such sets exist in any subhypergraph $H'\subseteq H^k_{n,p}$.

  First, let us fix such $F_1,\ldots, F_x$ and $Y$. Observe that the expected number of edges of the form $F_i\cup \{y\}$ in $H^k_{n,p}$ is exactly $xyp$.
  Therefore, by Lemma \ref{lemma:fancy Chernoff} we obtain
  $$\Pr[\# \text{ such edges }\geq xmp/2]\leq \left(\frac{2exyp}{xmp}\right)^{xmp/2}=\exp\left(-\frac{xmp}{2}\log \frac{m}{2ey}\right).$$

  By applying the union bound over all choice of $F_i$'s and $Y$ we obtain that the probability for having such sets which span at least $xmp/2$ edges of the form discussed above, is at most
  \begin{align*}
  \sum_{x=mp/2}^{m/10}&\binom{n}{k-1}^x\binom{n}{x}\exp\left(-\frac{xmp}{2}\log \frac{m}{2ey}\right)\\
  &\leq \sum_{x=mp/2}^{m/10}\left(\frac{en}{k-1}\right)^{kx}\left(\frac{en}{x}\right)^x \exp\left(-\frac{xmp}{2} \log\left(\frac{m}{2ex}\right)\right)\\
  &\leq \sum_{x=mp/2}^{m/10}\exp\left(kx\log\left(\frac{en}{k-1}\right)+x\log\left(\frac{en}{x}\right)-\frac{xmp}{2} \log\left({\frac{m}{2ex}}\right)\right)\\
  &\leq \sum_{x=mp/2}^{m/10}\exp\left(\left(k+1\right)x\log n-\frac{mpx}{2} \log\left({\frac{10}{2e}}\right)+O(1)\right)= o(1)
  \end{align*}
  where the last equality holds if we pick $p = C\log n/n$ where $C$ is a sufficiently large constant to satisfy $$\frac{mp}{2}\log \left(\frac{10}{2e}\right)>2(k+1)\log n$$

  Therefore, whp $B_\pi$ satisfies property 2.

  For property 3, let us fix $X\subseteq X_\pi$ and $Y\subseteq V_k$ of sizes $x$ and $y$ respectively where $m/10\leq x-1=y\leq m/2$. We now wish to establish an upper bound for the number of edges between them. Assume towards contradiction that $e_{B_{\pi}}(X,Y)>(1/2+\varepsilon/4)mpx$. Observe that this translates to the following: There exist $x$ disjoint sets $F_1,\ldots,F_x$, each of size exactly $k-1$ and a set $Y$ of size $x-1$, which is disjoint to all the $F_i$s, such that the number of edges in $H^k_{n,p}$, of the form $F_i\cup \{a\}$ where $a\in Y$, is larger than $(1/2+\varepsilon/4)mpx$. Let us show that whp $H^k_{n,p}$ has no such sets, thereby also guaranteeing that whp no such sets exist in any subhypergraph $H'\subseteq H^k_{n,p}$.

First, let us fix such $F_1,\ldots, F_x$ and $Y$. Observe that the expected number of edges of the form $F_i\cup \{y\}$ in $H^k_{n,p}$ is exactly $xyp$. Therefore, by Lemma \ref{lemma:Chernoff} we obtain
  $$\Pr[\# \text{ such edges } \geq (1/2+\varepsilon/4)mpx]\leq \exp\left(-\varepsilon^2xyp/40\right).$$

By applying the union bound we obtain that the probability to have such sets is at most
\begin{align*}
  \sum_{x=m/10}^{m/2}&\binom{n}{k-1}^x\binom{n}{x} \exp\left(-\varepsilon^2xyp/40\right)\\
  &\leq \sum_{x=m/10}^{m/2}n^{(k-1)x}n^{x} \exp\left(-\varepsilon^2xyp/40\right)\\
  &\leq \sum_{x=m/10}^{m/2}\exp\left((k-1)x\log n+x\log n-\varepsilon^2x^2p/40\right)=o(1)
\end{align*}
  where the last inequality holds if we pick $p=C\log n/n$ where $C$ is a sufficiently large constant to satisfy
  $$pm\varepsilon^2/400\geq 2k\log n.$$

  Therefore, whp $B_\pi$ satisfies property 3.

We can conclude that whp $B_\pi$ satisfies all three properties, and is $(\varepsilon / 2$, $p$)-pseudorandom. This completes the proof.
\end{proof}

Now that we know we can construct an $(\varepsilon / 2,p)$-pseudorandom bipartite graph $B_\pi$ from every subhypergraph $H$ with the properties outlined above, we will make use of the following lemma to show that every such $B_\pi$ must also contain a perfect matching. A similar proof appears in \cite{SV}.

\begin{lemma}
\label{lemma:resilience for bipartite}
 Every $(\varepsilon,p)$-pseudorandom bipartite graph contains a perfect matching.
\end{lemma}

\begin{proof} Let $G = (A \cup B, E)$ be an $(\varepsilon,p)$-pseudorandom bipartite graph with $|A| = |B| = m$. If $G$ does not contain a perfect matching, then it must violate the condition in Theorem \ref{thm:Hall}. That is, without loss of generality, there exists some \(X\subseteq A\) of size \(x\leq m/2\) and \(Y\subseteq B\) of size \(x-1\) such that \(N_{G}(X)\subseteq Y\). In particular, as \(\delta(G)\geq(1/2+\varepsilon)mp\) by property 1, it follows that $e_G(X,Y)\geq (1/2+\varepsilon)mpx$. In order to complete the proof we show that $G$ does not contain two such sets for all $1\leq x\leq m/2$.

We distinguish between three cases: First, assume $x \leq mp/2$.
As $|Y|\leq x<(1/2+\varepsilon)mp\leq\delta(G)$, it follows that \(N_{G}(X)\not\subseteq Y\).

Second, assume that
\(mp/2 \leq x \leq m/10 \).
By property 2, $e_G(X,Y)\leq mpx/2<(1/2+\varepsilon)mpx$, which is clearly a contradiction. Lastly, consider the case
\(m / 10\leq x \leq m/2\).
By property 3, $e_G(X,Y)\leq (1/2+\varepsilon/2)xmp<(1/2+\varepsilon)mpx$, which is also a contradiction.
This completes the proof.
\end{proof}

\section{Proof of Theorem \ref{main:codegree}}

Now we are ready to prove Theorem \ref{main:codegree}.
\begin{proof}
Let $k\in \mathbb{N}$, $\varepsilon>0$ and $p\geq C\log n/n$, for a sufficiently large $C$. Observe that, by Lemma \ref{degrees in random}, whp a hypergraph $H^k_{n,p}$ satisfies
$$(1-\varepsilon)np\leq \delta_{k-1}(H^k_{n,p})\leq \Delta_{k-1}(H^k_{n,p})\leq (1+\varepsilon)np.$$

Let $H\subseteq H^k_{n,p}$ be any subhypergraph with $\delta_{k-1}(H)\geq (1/2+\varepsilon)np$. We wish to show that $H$ contains a perfect matching.

To this end, as was previously explained in the outline, we will construct a bipartite graph in such a way that each perfect matching of this graph corresponds to a perfect matching of $H$.

To do so, let $\alpha>0$ where $(1-\alpha)(1/2+\varepsilon)\geq 1/2+\varepsilon/2$, and let us take a partitioning
  $[n]=V_1\cup \ldots \cup V_k$ into sets of the exact same size for which the following holds: For every subset $X\in \binom{[n]}{k-1}$ and for every $1\leq i\leq k$ we have $$d_H(X,V_i)\in (1\pm \alpha)\cdot \frac{d_H(X)}{k}.$$ In particular, for all $X\in \binom{[n]}{k-1}$ and all $1\leq i\leq k$, we have
  $$d_H(X,V_i)\geq (1/2+\varepsilon/2)mp,$$ where $m=\frac{n}{k}$. The existence of such a partitioning is guaranteed by Lemma \ref{random partition}.

  Next, let $H'$ be the resulting $k$-partite, $k$-uniform subhypergraph induced by the above partitioning. Recall that
$$\delta^*_{k-1}(H'):=\min\{d(X,V_i): X\in W_i, \text{ and }1\leq i\leq k \},$$
where $W_i=V_1\times \ldots \times V_{i-1}\times V_{i+1}\times \ldots \times V_k$.

Clearly, $\delta^*_{k-1}(H') \geq (1/2+\varepsilon/2)mp$. Therefore, Lemma \ref{Lemma:AlwaysPseudorandom} guarantees that there exists an auxiliary bipartite graph $B_\pi(H')$ (as defined in \ref{auxiliary-bipartite}) that is $(\varepsilon/4,p)$-pseudorandom. By Lemma \ref{lemma:resilience for bipartite}, such a $B_\pi$ would contain a perfect matching and therefore, by Remark \ref{remark:aux}, $H'$ must also contain a perfect matching. This completes the proof.
\end{proof}

{\bf Acknowledgments.} We are grateful to the referees for their valuable comments which were instrumental in revising this paper. The second author also greatly appreciates the support of the MIT math department in facilitating this Undergraduate Research
Opportunity.

\end{document}